\documentclass[11pt,letterpaper,reqno]{amsart}
\usepackage[T1]{fontenc}
\usepackage{lmodern,amsmath,amssymb,amsthm,mathtools,mathrsfs}
\usepackage{microtype}
\usepackage{xcolor}
\usepackage{etoolbox,needspace}
\usepackage[colorlinks=true,linkcolor=blue!50!black,citecolor=blue!50!black,urlcolor=blue!50!black]{hyperref}
\numberwithin{equation}{section}
\allowdisplaybreaks
\newtheorem{theorem}{Theorem}[section]
\newtheorem{proposition}[theorem]{Proposition}
\newtheorem{corollary}[theorem]{Corollary}
\newtheorem{lemma}[theorem]{Lemma}

\newtheoremstyle{mainresult}
  {0.75em}
  {0.75em}
  {\normalfont}
  {}
  {\bfseries}
  {.}
  {0.5em}
  {}
\theoremstyle{mainresult}
\newtheorem*{mainthmA}{Theorem A}
\newtheorem*{mainthmB}{Theorem B}
\newtheorem*{mainthmC}{Theorem C}
\makeatletter
\newcommand{\mainlabel}[2]{%
    \phantomsection
    \def\@currentlabel{#2}%
    \label{#1}}
\makeatother
\theoremstyle{definition}

\theoremstyle{remark}

\newtheorem*{shortremark}{Remark}
\newtheorem{example}[theorem]{Example}
\newif\ifcompactnextenv
\BeforeBeginEnvironment{mainthmA}{\Needspace{12\baselineskip}}
\BeforeBeginEnvironment{mainthmB}{\Needspace{10\baselineskip}}
\BeforeBeginEnvironment{mainthmC}{\Needspace{10\baselineskip}}
\BeforeBeginEnvironment{definition}{\Needspace{7\baselineskip}}
\BeforeBeginEnvironment{theorem}{\Needspace{7\baselineskip}}
\BeforeBeginEnvironment{proposition}{%
  \ifcompactnextenv
    \Needspace{3\baselineskip}\global\compactnextenvfalse
  \else
    \Needspace{6\baselineskip}
  \fi}
\BeforeBeginEnvironment{corollary}{\Needspace{6\baselineskip}}
\BeforeBeginEnvironment{lemma}{%
  \ifcompactnextenv
    \Needspace{3\baselineskip}\global\compactnextenvfalse
  \else
    \Needspace{6\baselineskip}
  \fi}
\BeforeBeginEnvironment{observation}{\Needspace{6\baselineskip}}
\BeforeBeginEnvironment{application}{\Needspace{6\baselineskip}}
\BeforeBeginEnvironment{problem}{\Needspace{6\baselineskip}}
\BeforeBeginEnvironment{remark}{\Needspace{5\baselineskip}}
\BeforeBeginEnvironment{shortremark}{%
  \ifcompactnextenv
    \Needspace{3\baselineskip}\global\compactnextenvfalse
  \else
    \Needspace{7\baselineskip}
  \fi}
\BeforeBeginEnvironment{example}{\Needspace{5\baselineskip}}
\BeforeBeginEnvironment{proof}{%
  \ifcompactnextenv
    \Needspace{2\baselineskip}\global\compactnextenvfalse
  \else
    \Needspace{4\baselineskip}
  \fi}
\makeatletter
\patchcmd{\section}{.7\linespacing\@plus\linespacing}{\baselineskip}{}{}
\makeatother
\pretocmd{\subsection}{\Needspace{5\baselineskip}}{}{}
\newcommand{\R}{\mathbb R}
\newcommand{\C}{\mathbb C}
\newcommand{\HH}{\mathbb H}
\newcommand{\Z}{\mathbb Z}
\newcommand{\T}{\mathbb T}
\newcommand{\Sphere}{\mathbb S}
\newcommand{\Vol}{\operatorname{Vol}}
\newcommand{\AW}{\mathcal A}
\title[Minimal products via weighted spheres]
{Minimal Products
\\
via Weighted Spheres}
\author{Yongsheng Zhang}
\address{Academy for Multidisciplinary Studies, Capital Normal University,
Beijing 100048, P. R. China}
\email{yszhang@cnu.edu.cn}
\date{October 29, 2026}
\hypersetup{
  pdftitle={Minimal Products via Weighted Spheres},
  pdfauthor={Yongsheng Zhang},
  pdfsubject={Minimal submanifolds and weighted coefficient spheres},
  pdfkeywords={
        minimal submanifold, spherical minimal product,
        weighted coefficient sphere, exact composition,
        Hsiang--Lawson reduction, equivariant doubling,
        Gaussian transfer, self-shrinker
        }}

\begin{document}

    \begin{abstract}
        We establish an exact composition principle for spherical immersions
            coupled through coefficient profiles and norm-preserving bilinear maps.
        Under explicit mixed orthogonality,
            the spherical mean curvature vector splits into orthogonal input and profile contributions.
        The composition is minimal if and only if every input is minimal
            and the profile is $W$-minimal for the induced monomial weight.
        Tensor iteration allows weighted profiles to serve as inputs in further compositions.

        At fixed weights, Hsiang--Lawson reduction and Kapouleas--McGrath doubling
            yield closed embedded $W$-minimal profiles.
        When the effective exponents tend to infinity
            in positive limiting proportions,
            Gaussian transfer produces interior coefficient profiles
            from closed embedded self-shrinkers
            that are nondegenerate modulo rotations.
        The number of distinct corresponding profiles is bounded below
            by the Lusternik--Schnirelmann category of the rotation orbit.
    \end{abstract}
    \maketitle

\section{Introduction}\label{intr}

        Minimal submanifolds in spheres form a central subject in geometry,
            linking classical differential geometry and geometric measure theory with group actions,
            nonlinear analysis, and spectral theory;
        see \cite{Sim68,Tak66,Law70,HL71,Pit81,MN14,KM23}.
        Methods of construction range from geometric reflection and symmetry reduction
            to min--max theory and analytic gluing.

        Given minimal immersions $f_i:M_i^{k_i}\to\Sphere^{n_i}$, $k_i>0$,
            the classical minimal product uses the constants $\sqrt{k_i/(\sum_j k_j)}$ \cite{Xin03,TZ20,CH18}.
        Inspired by spiral minimal products \cite{LZ26},
            we replace these constant coefficients by a varying profile while retaining minimality.
        Under a mixed orthogonality condition, we prove that the composition is minimal precisely
            when the inputs are minimal and the profile is minimal
            for a monomial weight determined by the input dimensions.
        This provides a general framework with several interfaces for constructing spherical minimal submanifolds
            through geometric and variational methods on weighted coefficient spheres.

        To describe the orthogonality condition,
            we specify the coefficient modules and the composition.
        All vector spaces below are Euclidean, direct sums are orthogonal,
            and $S(E)$ denotes the unit sphere of $E$;
        tensor products carry the product Euclidean inner product and are taken over $\R$.
        For $1\leq i\leq m$, choose a bilinear map
            $\mu_i:E_i\times C_i\to Z_i$ satisfying $|\mu_i(v,c)|=|v|\,|c|$.
        We call each of these an \emph{orthogonal coefficient module}.
        Scalar multiplication over $\R$,
        $\C$ or $\mathbb H$, and the real tensor map $(v,c)\mapsto v\otimes_{\R} c$ are standard examples.
        Let $f_i:M_i^{k_i}\to S(E_i)$ be immersions with $k_i\geq 0$, and define
    \[
        \mathcal S
            =S\bigl(\bigoplus C_i\bigr),
        \quad
        \mathcal T
            =S\bigl(\bigoplus Z_i\bigr),
        \quad
        \mathcal S^\circ
        =\{z=(z_i)_i\in\mathcal S:z_i\neq 0\ \text{for all }i\}.
    \]
        A profile with image in $\mathcal S^\circ$ will be called \emph{interior};
        unless stated otherwise, every profile below is interior.
        For an immersed profile $\Psi=(c_i)_i:P^r\to\mathcal S^\circ$, $r\geq 0$,
            and $x=(x_i)_i\in\prod_i M_i$, define the composition
    \begin{equation}\label{cmpmap}
        F:P\times\prod M_i
        \longrightarrow\mathcal T,
        \qquad
        F(p,x)
        =\bigl(\mu_i(f_i(x_i),c_i(p))\bigr)_i.
    \end{equation}
        To make the profile and input tangent directions orthogonal, we require
    \begin{equation}\label{pmix}
        \langle\mu_i(df_i(X),c_i),\mu_i(f_i,dc_i(U))\rangle=0
        \quad
        \text{ for }X\in TM_i,\ U\in TP.
    \end{equation}
        This mixed orthogonality condition is automatic for real scalar and tensor modules;
            for the remaining modules we use the test \eqref{hortest}, which is independent of the profile.
        Let $g_i=f_i^*g_{S(E_i)}$, $g_\Psi=\Psi^*g_{\mathcal S}$,
            and $W(z)=\prod_i |z_i|^{k_i}$.
        We use the trace convention for mean curvature.

    \mainlabel{maina}{A}
    \begin{mainthmA}[Exact composition principle]
        Assume \eqref{pmix}.
        Then $F$ is an immersion, with
    \begin{equation}\label{inmt}
        g_F=g_\Psi\oplus\bigoplus |c_i|^2g_i,\qquad
        d\Vol_{g_F}=W(\Psi)\,d\Vol_{g_\Psi}\prod d\Vol_{g_i}.
    \end{equation}
        It is minimal in $\mathcal T$ if and only if every input $f_i$ is minimal in $S(E_i)$ and
    \begin{equation}\label{main}
        \mathbf H_\Psi^{\mathcal S}
        =\bigl(\nabla^{\mathcal S}\log W\bigr)^\perp,
    \end{equation}
        where ${}^\perp$ denotes projection to the spherical normal bundle of the profile.
    \end{mainthmA}

        We call solutions of \eqref{main} \emph{$W$-minimal},
            namely stationary points of the functional
            $\AW_W(\Psi)=\int_P W(\Psi)\,d\Vol_{g_\Psi}$.
        On $\mathcal S^\circ$ the weight $W$ is positive.
        Thus, for $r\geq 1$, \eqref{main} is the minimal submanifold equation
            for $W^{2/r}g_{\mathcal S}$;
        for $r=0$, the profile consists of critical points of $W$.
        We prove Theorem \ref{maina} in Theorem \ref{comp}
            by an orthogonal mean curvature splitting.
        Under tensor iteration, weighted profiles can serve as inputs at the next stage;
            see Proposition \ref{itercmp}.

        In \cite{LZ26}, we treated two complex scalar modules
            with joint coefficient sphere $\Sphere^3$ and $\C$-horizontal inputs
            (called spherical $\mathscr C$-total reality there),
            so that \eqref{pmix} holds.
        The resulting $W$-geodesic flow is Liouville integrable:
            its energy and two Noether momenta give a complete description of the profiles,
            including their geometry, closing, and embeddedness.
        For the general coefficient structures of Theorem \ref{maina},
            we seek closed $W$-minimal profiles that yield closed minimal compositions.

    \Needspace{5\baselineskip}
        At fixed weights, Hsiang--Lawson reduction and phase closing produce closed coefficient tori
            whose spherical compositions have linear volume growth (Theorem \ref{urnk}),
        while the Kapouleas--McGrath doubling method \cite{KM23} produces the following surface families.

    \mainlabel{mainb}{B}
    \begin{mainthmB}[Weighted doubling profiles]
        For every pair of positive integers $k_1,k_2$,
            let $\mathcal S=S(\C\oplus\C)=\Sphere^3$ and $W(z_1,z_2)=|z_1|^{k_1}|z_2|^{k_2}$.
        Then $\mathcal S^\circ$ contains a sequence of closed embedded $W$-minimal surfaces
            with genera tending to infinity,
        uniformly bounded weighted areas, and images contained in a fixed compact subset of $\mathcal S^\circ$.
    \end{mainthmB}

        We next let the effective exponents grow.
        On the positive orthant
        $B=\Sphere^{m-1}_{++}:=\{r\in\Sphere^{m-1}:r_i>0\ \text{for all }i\}$,
        take $W_j(r)=\prod_i r_i^{\delta_i^{(j)}}$, where
        $D_j=\sum_i \delta_i^{(j)}\to\infty$ and $\delta_i^{(j)}/D_j\to\lambda_i>0$.
        In normal coordinates centered at the balanced maximizer,
            the normalized weight at scale $D_j^{-1/2}$ converges to $e^{-|y|^2}$,
        whose Euler--Lagrange equation is the self-shrinker equation
            $\mathbf H=-2y^\perp$ \cite{Hui90,Smo05,CM12}.
        A closed embedded solution is a \emph{Gaussian seed}
            when its normal Jacobi fields are precisely those induced by rotations.
        Lemma \ref{gmb} reduces the perturbation problem
            to a finite-dimensional weighted area functional on the compact rotation orbit,
        which leads to the following transfer theorem.

    \mainlabel{mainc}{C}
    \begin{mainthmC}[Gaussian transfer]
        Let $\Sigma^q\subset\R^{m-1}$, $1\leq q<m-1$, be a Gaussian seed and $H\leq O(m-1)$ its stabilizer.
        For all sufficiently large $j$,
            there are at least $\operatorname{cat}(O(m-1)/H)$ closed embedded $W_j$-minimal submanifolds
            $Q_{j,\ell}\cong\Sigma$ in $B$, with pairwise distinct unparametrized images.
        Here $\operatorname{cat}$ denotes the unreduced Lusternik--Schnirelmann category,
            summed over connected components, with $\operatorname{cat}(\mathrm{point})=1$.
    \end{mainthmC}

        Round Gaussian spheres and standard Veronese projective planes provide seeds
            (Lemma \ref{gseed} and Proposition \ref{prjseed}).
        Theorem \ref{mainc} supplies real scalar profiles,
            while Example \ref{normcmp} supplies more general couplings.
        Coupling these profiles with independent minimal inputs yields Proposition \ref{coupled};
        ordinary doubling surfaces yield families in fixed codimension seven
            (Corollary \ref{codimseven}).

        Companion manuscripts develop several directions closely related to the framework of this paper:
            assembly from a common source \cite{SS26},
            anisotropic Gaussian transfer and hypercone applications \cite{Hyp26},
            and weighted special Legendrian profiles \cite{SL26}.

        This paper is organized as follows.
        We establish the composition and tensor iteration in \S~\ref{fram},
            treat closed profiles at fixed weights in \S~\ref{part},
            and derive Gaussian transfer and its applications in \S~\ref{gaus}.
        Appendix \ref{persist} proves the persistence lemma for Gaussian transfer.

\Needspace{6\baselineskip}
\section{Composition and iteration}\label{fram}

\subsection{Exact composition}

        We retain the notation of Theorem \ref{maina}
            and let $\Delta_\Psi$ denote the nonpositive Laplacian of $g_\Psi$.
        For a scalar- or vector-valued map $u$ on $P^r$, set
    \[
        \Delta_W u
        =W(\Psi)^{-1}\operatorname{div}_{g_\Psi}
          \bigl(W(\Psi)\nabla^{g_\Psi}u\bigr)
        =\Delta_\Psi u
          +du\bigl(\nabla^{g_\Psi}\log(W\circ\Psi)\bigr).
    \]
        The $W$-weighted mean curvature vector is
            $ \mathbf H_W(\Psi) = \mathbf H_\Psi^{\mathcal S}-(\nabla^{\mathcal S}\log W)^\perp $.
        If $K=\sum_j k_j$, direct differentiation shows
            $ (\nabla^{\mathcal S}\log W)_i = k_i|c_i|^{-2}c_i-K c_i $.

        Combining Takahashi's identity
            $\Delta_\Psi \Psi=\mathbf H_\Psi^{\mathcal S}-r\Psi$ \cite{Tak66}
            with the formulas above, we obtain
    \begin{equation}\label{wdef}
        (\mathbf H_W(\Psi))_i
        =\Delta_W c_i
         +\bigl(r+\sum\nolimits_j k_j\bigr)c_i
         -k_i|c_i|^{-2}c_i.
    \end{equation}

    \begin{theorem}[Mean curvature splitting]\label{comp}
        Assume \eqref{pmix} holds for $1\leq i\leq m$.
        Then $F$ is an immersion with metric and volume form given by \eqref{inmt}.
        Set $\rho_i=|c_i|$.
        For each $i$,
            the $i$-th block of its spherical mean curvature vector splits into two orthogonal terms:
    \begin{equation}\label{lapl}
        (\mathbf H_F)_i
        =\mu_i\bigl(f_i,(\mathbf H_W(\Psi))_i\bigr)
         +\rho_i^{-2}\mu_i(\mathbf H_{f_i},c_i).
    \end{equation}
        Thus, $F$ is minimal if and only if every input is minimal and $\Psi$ is $W$-minimal.
    \end{theorem}

    \begin{proof}
        By polarizing the norm identity, we obtain
    \[
        \langle\mu_i(v,c),\mu_i(w,c)\rangle
        =|c|^2\langle v,w\rangle,
        \qquad
        \langle\mu_i(v,c),\mu_i(v,d)\rangle
        =|v|^2\langle c,d\rangle.
    \]
        Together with \eqref{pmix}, these identities imply \eqref{inmt} and show that $F$ is an immersion.
        Let $\Delta_F$ and $\Delta_i$ denote the Laplacians of $g_F$ and $g_i$.
        The warped product formula reads $\Delta_F u=\Delta_W u$ for $u\in C^\infty(P)$, and hence
    \[
        \Delta_F \mu_i(f_i,c_i)
        =\mu_i(f_i,\Delta_W c_i)
         +\rho_i^{-2}\mu_i(\Delta_i f_i,c_i).
    \]
        Since $\Delta_i f_i=\mathbf H_{f_i}-k_i f_i$, comparison with \eqref{wdef} proves \eqref{lapl}.

        It remains to prove orthogonality.
        Differentiate \eqref{pmix} in a local $g_i$-orthonormal frame and trace.
        The two terms $k_i\langle c_i,dc_i(U)\rangle$ cancel, leaving
        \[
            \langle\mu_i(\mathbf H_{f_i},c_i),
            \mu_i(f_i,dc_i(U))\rangle=0.
        \]
        Taking its weighted $P$-divergence and using $\langle\mathbf H_{f_i},f_i\rangle=0$, we obtain
        \[
            \langle\mu_i(\mathbf H_{f_i},c_i),
            \mu_i(f_i,\Delta_W c_i)\rangle=0.
        \]
        Equation \eqref{wdef} therefore makes the two terms in \eqref{lapl} orthogonal.
        Finally, the norm identity and $\rho_i>0$ make both relevant linear maps injective.
        Thus \eqref{lapl} vanishes exactly when $\mathbf H_W(\Psi)=0$ and every $\mathbf H_{f_i}=0$.
    \end{proof}

        We now derive a condition on $f$ and $\mu$ that implies \eqref{pmix} for every coefficient map.
        Set $d=\dim C$, fix an orthonormal basis $(e_\alpha)_{\alpha=1}^d$ of $C$,
            and define $A_\alpha v=\mu(v,e_\alpha)$.
        Polarization of the norm identity shows that $A_\alpha^*A_\beta+A_\beta^*A_\alpha=2\delta_{\alpha\beta}I$,
        where ${}^*$ denotes the Euclidean adjoint.
        Thus $J_{\alpha\beta}=A_\alpha^*A_\beta=-J_{\beta\alpha}$ for $1\leq \alpha<\beta\leq d$.
        The following conditions are equivalent:
    \begin{equation}\label{hortest}
        \begin{split}
        \langle df(X),J_{\alpha\beta}f\rangle&=0
        \quad\text{for all }X,\ \alpha<\beta,
        \\
        \langle\mu(df(X),a),\mu(f,b)\rangle&=0
        \quad\text{for all }X,\ a,b\in C.
        \end{split}
    \end{equation}
        Thus either condition in \eqref{hortest} guarantees \eqref{pmix} for every coefficient map $c:P\to C$.
        We call $f$ $\mu$-horizontal if it satisfies \eqref{hortest}.
        The standard modules are listed below.
        For scalar modules we use the convention $\mu(v,c)=vc$,
            while $k=\dim M$, $d=\dim C$,
            and the last column is the effective exponent after adding the full phase sphere $S(C)$,
        as in Proposition \ref{itercmp}.
    \begin{center}
    \small
    \setlength{\tabcolsep}{4pt}
    \begin{tabular}{@{}llll@{}}
        \hline
        Module \ \ \ \ \ \  \ \ \ \ \   \ \ \ \ \ \ \ \  \ \ & \shortstack[l]{$S(C)$\  \ \ \ \ }
        & Input condition \ \ \ \ \ \ \ \ \ \ \ \  & \shortstack[l]{Full phase exponent}
        \\
        \hline
        Real scalars & $\Sphere^0$ & Automatic & $k$
        \\
        Complex scalars & $\Sphere^1$ & $\langle df,f\mathbf i\rangle=0$ & $k+1$
        \\
        Quaternionic scalars & $\Sphere^3$
        & $\langle df,f\xi\rangle=0$, $\xi=\mathbf i,\mathbf j,\mathbf k$ & $k+3$
        \\
        Real tensor $v\otimes_{\R} c$ & $\Sphere^{d-1}$ & Automatic & $k+d-1$
        \\
        \hline
    \end{tabular}
    \end{center}
        We call the complex and quaternionic conditions $\C$- and $\HH$-horizontality, respectively.
        The former is spherical $\mathscr C$-total reality in the terminology of \cite{LZ26}.
        Complexifying or quaternionifying a real spherical immersion produces automatic horizontal examples,
        since scalar multiplication restricts on $E\otimes_{\R}1$ to the real tensor map.

\subsection{Tensor iteration}\label{itersec}

        We regard each inner profile as a group of blocks and let the outer profile weight these groups.
        Here $i$ labels groups and $j$ their blocks.

    \begin{proposition}[One-step tensor iteration]\label{itercmp}
        For each $1\leq i\leq m$, let
            $\Psi_i=(\psi_{ij})_{j=1}^{b_i}:M_i^{r_i}\to S(E_i)$,
            where $E_i=\bigoplus_{j=1}^{b_i} E_{ij}$,
        be an immersed profile of dimension $r_i\geq 0$ with nonvanishing blocks.
        Given block exponents $k_{ij}\geq 0$, define
    \[
        W_i(u_i)
        =\prod_j |u_{ij}|^{k_{ij}},
        \qquad
        K_i
        =\sum_j k_{ij},
        \qquad
        \delta_i
        =r_i+K_i.
    \]
        Thus $K_i$ is the homogeneity degree of the inner weight $W_i$,
            while $\delta_i$ is the effective exponent of the $i$-th group.
        An immersed outer profile $a=(a_i)_i:P^s\to S(\bigoplus_i C_i)$,
            with $s\geq 0$ and $\rho_i=|a_i|>0$, then defines
    \enlargethispage{\baselineskip}
    \begin{equation}\label{itermap}
        \Theta\colon P\times\prod_i M_i
        \longrightarrow
        S\bigl(\bigoplus_{i,j}E_{ij}\otimes_{\R}C_i\bigr),
        \,\text{by }\,
        \Theta(y,x)
        =(\psi_{ij}(x_i)\otimes_{\R}a_i(y))_{i,j}.
    \end{equation}
        Set $g_a=a^*g_{S(\oplus_i C_i)}$,
            $W(z)=\prod_{i,j}|z_{ij}|^{k_{ij}}$
            and $V(a)=\prod_i\rho_i^{\delta_i}$.
        For $s=0$, we use the convention stated after Theorem \ref{maina}.
        Then $\Theta$ is an immersion, with
    \begin{equation}\label{iterme}
        \begin{split}
            &\ \ \;\;\;\; g_\Theta
                =g_a\oplus\bigoplus_i \rho_i^2g_i,
                  \qquad g_i=\Psi_i^*g_{S(E_i)},
                \\
            &W(\Theta)\,d\Vol_{g_\Theta}
                =V(a)\,d\Vol_{g_a}
                  \prod_i\bigl(W_i(\Psi_i)\,d\Vol_{g_i}\bigr).
        \end{split}
    \end{equation}
        For each group $i$, the weighted mean curvature splits orthogonally as
    \begin{equation}\label{iterdef}
        (\mathbf H_W(\Theta))_i
        =\Psi_i\otimes_{\R}(\mathbf H_V(a))_i
         +\rho_i^{-2}\mathbf H_{W_i}(\Psi_i)\otimes_{\R} a_i.
    \end{equation}
        Thus weighted minimality of $\Theta$ is equivalent to that of $a$ and the $\Psi_i$.
    \end{proposition}

    \begin{proof}
        The metric follows as in Theorem \ref{comp}.
        Combining the volume factors $\rho_i^{r_i}$ with the weight factors $\rho_i^{K_i}$
            proves \eqref{iterme}.
        We also have
    \[
        \Delta_W(\Psi_i\otimes_{\R} a_i)
        =\Psi_i\otimes_{\R}\Delta_V a_i
         +\rho_i^{-2}(\Delta_{W_i}\Psi_i)\otimes_{\R} a_i.
    \]
        Substitution into \eqref{wdef} proves \eqref{iterdef}.
        Since $\mathbf H_{W_i}(\Psi_i)\perp\Psi_i$,
            the two terms in \eqref{iterdef} are orthogonal, proving the equivalence.
    \end{proof}

    \begin{example}[Coupling through block norms]\label{normcmp}
        Let $f_i:M_i^{k_i}\to S(E_i)$, $k_i>0$, be spherical minimal immersions,
            possibly obtained by earlier compositions.
        For an immersed profile $\Psi=(\rho_i)_i:P^r\to\Sphere^{m-1}_{++}$, consider
            $F(p,x)=(\rho_i(p)f_i(x_i))_i$.
        By Theorem \ref{comp}, it is minimal precisely when $\Psi$ is $V$-minimal for $V=\prod_i \rho_i^{k_i}$.
        It is embedded when $\Psi$ and every $f_i$ are embeddings.
        The zero-dimensional choice $\rho_i^2=k_i/(\sum_j k_j)$ recovers the constant minimal product;
        for $m\geq 3$ one may instead use a closed $V$-minimal profile of dimension $1\leq r\leq m-2$,
        whenever one is available.
        We will see that Corollary \ref{gint} supplies such profiles in every one of these dimensions
            for all sufficiently large balanced exponents.
    \end{example}
\subsection{Embedded compositions and congruence}

    \begin{proposition}[Embedded tensor compositions]\label{etens}
        Let $M_i^{k_i}\subset S(E_i)$ be closed embedded minimal submanifolds satisfying $-M_i=M_i$.
        Set $W(z)=\prod_i |z_i|^{k_i}$
            and let $P\subset S(\bigoplus_i C_i)$ be a closed embedded $W$-minimal profile
            with nonvanishing blocks, invariant under the coordinate sign group $G=(\Z_2)^m$.
        Then the real tensor map induces a minimal embedding
    \[
        \bar F:\bigl(P\times\prod M_i\bigr)/G
        \hookrightarrow S\bigl(\bigoplus E_i\otimes_{\R} C_i\bigr),
        \qquad
        [z,x]
        \longmapsto(x_i\otimes_{\R} z_i)_i,
    \]
        where $\varepsilon=(\varepsilon_i)_i\in G$ acts by $(z_i,x_i)\mapsto(\varepsilon_iz_i,\varepsilon_ix_i)$.
    \end{proposition}
    \begin{proof}
        By Theorem \ref{comp}, this is a $G$-invariant minimal immersion, and the $G$-action is free.
        For each $i$,
        $x_i\otimes_{\R}z_i=x'_i\otimes_{\R}z'_i
        \Longleftrightarrow
        \exists\,\varepsilon_i\in\{\pm1\}:\;
        x'_i=\varepsilon_i x_i,\ z'_i=\varepsilon_i z_i$,
        so the quotient immersion is injective and embedded.
    \end{proof}

    \begin{proposition}[Congruence of compositions with round sphere inputs]\label{congr}
        Let $m\geq 3$, $k_1,\ldots,k_m\geq 1$ be integers,
            and $Q,Q'\subset\Sphere^{m-1}_{++}$ closed embedded submanifolds of positive dimension.
        Set $D=\sum_i k_i$ and
    \[
        M_Q=\bigl\{(r_ix_i)_i:r\in Q,\ x_i\in\Sphere^{k_i}\bigr\}
        \subset\Sphere^{D+m-1}.
    \]
        Then $M_Q$ and $M_{Q'}$ are congruent under an ambient isometry
            if and only if $Q'=\sigma Q$ for a coordinate permutation $\sigma$
            satisfying $k_{\sigma(i)}=k_i$ for every $i$.
    \end{proposition}

    \begin{proof}
        Let $E_i=\R^{k_i+1}$, and let $e^{tA}$, $A=(A_{ij})$,
            be any rotation flow preserving $M_Q$.
        For each $i$, choose $r\in Q$ minimizing $r_i$.
        The derivative of the $i$-th block norm along this flow vanishes at $(r_jx_j)_j$, so
            $\sum_{j\neq i}r_j\langle x_i,A_{ij}x_j\rangle=0$
            for all $x_j\in S(E_j)$.
        By replacing one $x_j$ by $-x_j$, we see $A_{ij}=0$ for $j\neq i$.
%
        Thus, in either stabilizer, the component containing the identity is $\prod_i SO(E_i)$.
        A congruence conjugates these groups, hence permutes their irreducible blocks by some $\sigma$, namely $Q'=\sigma Q$.
        The converse is immediate.
    \end{proof}


\section{Closed profiles at fixed weights}\label{part}

        We first keep the weights fixed.
        Quaternionic phase closing produces many coefficient tori,
            while weighted doubling supplies profiles of unbounded topology.
        Circle saturation then combines the two mechanisms,
            and all resulting profiles remain in $\mathcal S^\circ$.
        Throughout this section, $k_1$ and $k_2$ are positive integers.

\subsection{Quaternionic phase closing}\label{phclose}

        Let $\mathcal S=S(\HH\oplus\HH)=\Sphere^7$
            and set $W_0(a,b)=|a|^{k_1}|b|^{k_2}$ on $\mathcal S^\circ$.
        The block norms determine a unique magnitude angle $s\in(0,\frac{\pi}{2})$
            by $|a|=c=\cos s$ and $|b|=d=\sin s$.
        Dividing $p$ and $q$ by their common divisor merely reparametrizes the orbit circle
            and leaves its foliation unchanged.
        We therefore fix $p,q\in\Z\setminus\{0\}$ with $\gcd(|p|,|q|)=1$.
        With this normalization, the circle acts freely on $\mathcal S^\circ$ by
    \begin{equation}\label{uact}
        g_\vartheta(a,b)
        =(ae^{\mathbf i p\vartheta},be^{\mathbf i q\vartheta}),
        \qquad
        \vartheta\in\R/(2\pi\Z),
    \end{equation}
        where $\mathbf i$ is the standard quaternionic unit.
        The generating vector field is $(pa\mathbf i,qb\mathbf i)$,
            so the orbit has round length $2\pi\ell(s)$,
            where $\ell(s)=\sqrt{p^2c^2+q^2d^2}$.
        Since $W_0$ is constant on orbits,
            fiber integration leaves the density $2\pi A(s)$ with $A(s)=c^{k_1}d^{k_2}\ell(s)$.
        Let $\operatorname{pr}:\mathcal S^\circ\to\mathcal B$ be the orbit projection
            and let $g_{\mathcal B}$ be the round quotient metric.
        For $\bar P^r\subset\mathcal B$ and its saturation $P=\operatorname{pr}^{-1}(\bar P)$,
            the Hsiang--Lawson theorem for free actions \cite[Theorem 2]{HL71}, applied to
            $W_0^{2/(r+1)}g_{\mathcal S}$, reduces weighted minimality to the quotient data below.
        In this metric, the orbit length and quotient metric are
            $2\pi\ell W_0^{1/(r+1)}$ and $W_0^{2/(r+1)}g_{\mathcal B}$,
        respectively; hence
    \begin{equation}\label{hlred}
        \AW_{W_0}(P)
        =\int_{\bar P} 2\pi A\,d\Vol_{g_{\mathcal B}}
        =\Vol_{g_{\mathcal B}^{(r)}}(\bar P),
        \qquad
        g_{\mathcal B}^{(r)}
        :=(2\pi A)^{2/r}g_{\mathcal B},
    \end{equation}
        and $P^{r+1}$ is $W_0$-minimal precisely when $\bar P^r$ is minimal in $(\mathcal B,g_{\mathcal B}^{(r)})$.

        For the standard quaternionic unit $\mathbf j\perp\mathbf i$,
            consider the three-dimensional horizontal submanifold
    \begin{equation}\label{hslice}
        L=\{(ce^{\mathbf j\varphi_1},de^{\mathbf j\varphi_2}):
        0<s<\frac{\pi}{2},\ \varphi_1,\varphi_2\in\R/(2\pi\Z)\}.
    \end{equation}
        Its coordinate fields are orthogonal to the orbit field,
            so $L$ is horizontal and its quotient metric is
            $ds^2+c^2\,d\varphi_1^2+d^2\,d\varphi_2^2$.
        At fixed $s$, two phase pairs represent the same quotient point exactly when their difference lies in
    \begin{equation}\label{phlat}
        \Lambda_{p,q}
        =\operatorname{span}_{\Z}
         \bigl\{(2\pi,0),(0,2\pi),(\pi p,\pi q)\bigr\}.
    \end{equation}
         For $x\in L$, coprimality implies
    $
        g_\vartheta x\in L
        \Longleftrightarrow
        p\vartheta,q\vartheta\in\pi\Z
        \Longleftrightarrow
        \vartheta\in\pi\Z.
    $
        Thus $\operatorname{pr}|_L$ is the double cover identifying
            $(a,b)$ with $((-1)^pa,(-1)^qb)$, and
        $\bar L:=\operatorname{pr}(L)
            \cong(0,\frac{\pi}{2})\times(\R^2/\Lambda_{p,q})$ is three-dimensional.
        The two phase translations are induced by left multiplications
            $(a,b)\mapsto(e^{\mathbf j\tau_1}a,e^{\mathbf j\tau_2}b)$,
        which preserve $W_0$ and commute with the circle action \eqref{uact}.

    \compactnextenvtrue
    \begin{lemma}[Totally geodesic quotient]\label{hfix}
        For every $1\leq r\leq 6$, $\bar L$ is totally geodesic in $(\mathcal B,g_{\mathcal B}^{(r)})$.
    \end{lemma}

    \begin{proof}
        The involution $T(a,b)=(\mathbf ja(-\mathbf j),\mathbf jb(-\mathbf j))$
            preserves $W_0$ and $\ell$ and conjugates $g_\vartheta$ to $g_{-\vartheta}$,
        so it induces an isometry of every $g_{\mathcal B}^{(r)}$.
        The fixed component meeting $\bar L$ is precisely $\bar L$:
        if $Tx=g_\vartheta x$, then $g_{\vartheta/2}x$ is $T$-fixed.
        Hence $\bar L$ is totally geodesic.
    \end{proof}

        By total geodesicity and \eqref{hlred}, a minimal $r$-submanifold of
            $(\bar L,g_{\mathcal B}^{(r)}|_{\bar L})$ 
                has a $W_0$-minimal $(r+1)$-dimensional circle saturation
       ($r=2$ used for the three-folds in \S \ref{satsec}).
        For $r=1$, we may therefore look for quotient geodesics inside $\bar L$.
        Dropping the constant factor $(2\pi)^2$, the reduced metric is
    \begin{equation}\label{umet}
        h_{p,q}=A(s)^2
        \bigl(ds^2+c^2\,d\varphi_1^2+d^2\,d\varphi_2^2\bigr).
    \end{equation}

    \medskip
    \noindent\emph{Tori at constant $s$.}
        We saturate the circle
            $\{(ce^{\mathbf j\varphi},d):\varphi\in\R/(2\pi\Z)\}\subset L$
            to obtain an embedded torus $\mathcal T_s^{p,q}$.
        Its weighted area is proportional to $\mathscr B(s):=A(s)\cos s=c^{k_1+1}d^{k_2}\ell(s)$.
        Lemma \ref{uaxs} shows that $\mathscr B$ has a unique nondegenerate maximum at $s_0$.
        At this value, the phase circle is an $h_{p,q}$-geodesic,
            so Lemma \ref{hfix} and Hsiang--Lawson reduction make $\mathcal T_{s_0}^{p,q}$ $W_0$-minimal.

        For an embedded quotient curve $\bar\gamma\subset\bar L$, \eqref{hlred} gives
    \[
        \AW_{W_0}\bigl(\operatorname{pr}^{-1}(\bar\gamma)\bigr)
        =2\pi\operatorname{Length}_{h_{p,q}}(\bar\gamma).
    \]

        A closed embedded quotient geodesic lifts to a two-dimensional coefficient torus in $\Sphere^7$.
        When $|p|=|q|=1$, we have $\ell\equiv 1$:
            the local reduced equation and phase integrals are exactly those of complex SMP \cite{LZ26},
        although global return is taken modulo $\Lambda_{p,q}$ rather than $2\pi\Z^2$.
        For $|p|\neq |q|$, the factor $\ell(s)^2$ changes the radial potential
            but not the Hamiltonian form of phase reconstruction.

    \medskip
    \noindent\emph{Radial oscillation and the period map.}
        We now fix $k_1,k_2,p,q$ and vary the conserved momentum.
        As in the phase reduction for complex SMP \cite[Sections 2--3]{LZ26},
        translations in $\varphi_1$ and $\varphi_2$ are isometries of $h_{p,q}$.
        Their Noether momenta are
    \[
        \zeta_1
        =A^2c^2\dot\varphi_1,
        \qquad
        \zeta_2
        =A^2d^2\dot\varphi_2.
    \]
        For an $h_{p,q}$-unit-speed geodesic, conservation of energy becomes
    \begin{equation}\label{uradial}
        A^2\dot s^2+\mathcal V_\zeta(s)
        =1,
        \qquad
        \mathcal V_\zeta(s)
        =\frac{\zeta_1^2}{A^2c^2}+\frac{\zeta_2^2}{A^2d^2}.
    \end{equation}
        For fixed $\zeta$, a \emph{regular well} is a component
            $(s_-,s_+)\Subset(0,\frac{\pi}{2})$ of $\{\mathcal V_\zeta<1\}$
            whose endpoints are simple roots of $\mathcal V_\zeta=1$.
        One \emph{radial cell} is $s_-\to s_+\to s_-$.
        Dividing $\dot\varphi_a$ by $\dot s$ on the outgoing leg
            and doubling for the return gives
    \begin{align}
        \Delta_1&=2\int_{s_-}^{s_+}
        \frac{\zeta_1\,ds}{Ac^2\sqrt{1-\mathcal V_\zeta}},
        &\Delta_2&=2\int_{s_-}^{s_+}
        \frac{\zeta_2\,ds}{Ad^2\sqrt{1-\mathcal V_\zeta}}.\label{phper}
    \end{align}
        Thus $\Pi(\zeta)=(\Delta_1(\zeta),\Delta_2(\zeta))$ records the phase advance per radial cell.
        Initial phases differ only by ambient phase translations, and the radial data return after each cell.
        Hence the radial closing order is the least $n\geq 1$ such that $n\Pi(\zeta)\in\Lambda_{p,q}$,
        as in \cite[Sections 2--3]{LZ26}.

    \Needspace{12\baselineskip}
    \begin{lemma}[A regular point of the period map]\label{uaxs}
        On any fixed branch of regular wells, the signed period map $\Pi$ is locally real analytic in the momenta.
        The function $\mathscr B(s)=A(s)\cos s$ has a unique nondegenerate maximum $\mathscr B_*=\mathscr B(s_0)$.
        Every $0<\zeta<\mathscr B_*$ determines a regular well at $(\zeta,0)$,
            and $d\Pi$ is invertible at some such axis point.
    \end{lemma}

    \begin{proof}
        We set $x=\sin^2s$.
        By direct differentiation, we obtain
    \begin{equation}\label{ucon}
        \frac{d^2}{dx^2}\log \mathscr B
        =-\frac{k_1+1}{2(1-x)^2}
         -\frac{k_2}{2x^2}
         -\frac{(q^2-p^2)^2}{2[p^2+(q^2-p^2)x]^2}<0.
    \end{equation}
        Since $\mathscr B$ vanishes at both endpoints, it has a unique nondegenerate maximum,
            and every $0<\zeta<\mathscr B_*$ determines a regular well on the axis $\zeta_2=0$.
        Local real analyticity of the period integrals, including across $\zeta_2=0$,
            follows from \cite[Lemma A.1]{LZ26}.
        On this axis,
    \[
        \Delta_1(\zeta,0)
        =2\zeta
         \int_{s_-}^{s_+}
         \frac{ds}{\cos s\sqrt{\mathscr B(s)^2-\zeta^2}}.
    \]
        Let $\kappa=-(\log\mathscr B)''(s_0)>0$.
        We compare the endpoint limits of $\Delta_1$:
    \begin{equation}\label{uend}
        \lim_{\zeta\downarrow 0}\Delta_1(\zeta,0)
        =\frac{\pi}{k_1+1},
        \qquad
        \lim_{\zeta\uparrow\mathscr B_*}\Delta_1(\zeta,0)
        =\frac{2\pi}{\cos s_0\sqrt{\kappa}}>\frac{\pi}{k_1+1}.
    \end{equation}
        The endpoint expansions $\mathscr B(s)=|p|s^{k_2}(1+O(s^2))$
            and $\mathscr B(\frac{\pi}{2}-y)=|q|y^{k_1+1}(1+O(y^2))$,
        together with the corresponding rescalings, show that the first endpoint contributes $o(1)$ and
    \[
        \lim_{\zeta\downarrow 0}\frac{\Delta_1(\zeta,0)}2
        =\int_1^\infty
         \frac{dt}{t\sqrt{t^{2(k_1+1)}-1}}
        =\frac{\pi}{2(k_1+1)}.
    \]
        The second limit in \eqref{uend} follows by applying
            the formula for coalescing turning points \cite[(A.1)]{LZ26} to $\mathscr B^2$,
        with parameter $\zeta^{-2}$ and weight $\sec s$.
        To compare the two limits, set $x_0=\sin^2s_0$
            and $\lambda=\frac{q^2x_0}{p^2(1-x_0)+q^2x_0}\in(0,1)$.
        The critical point equation reads $x_0=\frac{k_2+\lambda}{k_1+k_2+2}$, and \eqref{ucon} reduces to
    \[
        \cos^2s_0\,\kappa
        =2(k_1+2-\lambda)
         -\frac{2\lambda(1-\lambda)}{x_0}
        <2(k_1+2)\leq 4(k_1+1)^2.
    \]
        This proves the strict inequality in \eqref{uend}.
        Finally, $\Delta_1$ is even and $\Delta_2$ is odd in $\zeta_2$,
        so
    \[
        \det d\Pi(\zeta,0)
        =2\partial_\zeta\Delta_1(\zeta,0)
         \int_{s_-}^{s_+}
         \frac{ds}{Ad^2\sqrt{1-\frac{\zeta^2}{\mathscr B^2}}}.
    \]
        The integral is positive,
            while the distinct endpoint limits force $\partial_\zeta\Delta_1(\zeta,0)\neq 0$ somewhere.
        Thus $\det d\Pi$ is not identically zero on the connected branch containing this axis point,
        and analyticity makes its nonzero set open and dense there.
    \end{proof}

    \Needspace{12\baselineskip}
    \begin{theorem}[Closed toroidal profiles]\label{urnk}
        Fix nonzero coprime integers $p,q$.
        There exist $c_0,c_1,c_2>0$, an integer $N_0$,
            and $0<s_{\min}<s_{\max}<\frac{\pi}{2}$ such that the following holds.
        For every integer $N\geq N_0$, there are at least $\lfloor c_0N\rfloor$
            pairwise geometrically distinct closed embedded $W_0$-minimal tori
            $Q_{N,\nu}\subset\mathcal S^\circ$, invariant under \eqref{uact}.
        Their quotient curves have least radial closing order $N$.
        For each fixed $N$, their weighted areas are pairwise distinct and satisfy
            $c_1N\leq \AW_{W_0}(Q_{N,\nu})\leq c_2N$.
        Consequently, they are pairwise noncongruent under round isometries preserving $W_0$.
        Moreover, their radial coordinates lie in $[s_{\min},s_{\max}]$.
        All constants depend only on $k_1,k_2,p$,
        and $q$.

    \end{theorem}
    \begin{proof}
        \emph{Local inversion and lattice returns.}
        Fix a regular axis point $(\zeta_*,0)$ supplied by Lemma \ref{uaxs}
            and a neighborhood $U$ on which $\Pi$ is an analytic diffeomorphism.
        Choose a rectangle $R=[u,v]\times[-\delta,\delta]\Subset\Pi(U)$ and,
            after shrinking it if necessary, set $K=(\Pi|_U)^{-1}(R)$ so that $\zeta_1>0$ on $K$.
        Compactness gives $0<s_{\min}:=\min_Ks_-\leq s_-(\zeta)<s_+(\zeta)
            \leq\max_Ks_+=:s_{\max}<\frac{\pi}{2}$.
        For large $N$, set $\xi_\nu=(2\pi\nu,2\pi)$ or
            $(2\pi\nu+\pi p,\pi)$ according as $q$ is even or odd.
        These $\xi_\nu$ lie in $\Lambda_{p,q}$ and have first coordinates spaced by $2\pi$.
        Hence, for suitable $c_0>0$ and $N_0$,
            at least $\lfloor c_0N\rfloor$ of the points $\eta_{N,\nu}:=\xi_\nu/N$ lie in $R$
            whenever $N\geq N_0$.
        Their inverse images $\zeta_{N,\nu}:=(\Pi|_U)^{-1}(\eta_{N,\nu})$
            determine quotient geodesics closing after $N$ radial cells.

        \emph{Closing and embeddedness.} Set $P_q=2\pi$ for even $q$ and $P_q=\pi$ for odd $q$.
        By \eqref{phlat}, $\varphi_2\bmod P_q$ is well defined;
        along the geodesic it is strictly increasing
            because $\zeta_{N,\nu,2}>0$ and advances by $P_q$ over $N$ cells.
        It is a local diffeomorphism of circles of degree one, hence a diffeomorphism,
            so the quotient geodesic is embedded.
        A return after $n$ cells requires $nP_q/N\in P_q\Z$, so $N\mid n$.
        Thus the free circle lift is an embedded $W_0$-minimal torus of least radial closing order $N$.

        \emph{Area and volume.} Let $\mathscr L(\zeta)$ be the $h_{p,q}$-length of one radial cell.
        The radial action
            $\mathcal I(\zeta)=2\int_{s_-}^{s_+} A\sqrt{1-\mathcal V_\zeta}\,ds$
        satisfies $d\mathcal I=-\Pi\cdot d\zeta$, and $\mathscr L=\mathcal I+\zeta\cdot\Pi$.
        Thus $d\mathscr L=\zeta\cdot d\Pi$ and
            $\partial_{\eta_1}(\mathscr L\circ\Pi^{-1})=\zeta_1>0$
        for $\eta=\Pi(\zeta)$.
        For fixed $N$, $\eta_{N,\nu,2}$ is fixed and $\eta_{N,\nu,1}$ increases with $\nu$,
            so the areas $2\pi N\mathscr L(\zeta_{N,\nu})$ are pairwise distinct.
        Bounds for $\mathscr L$ on $K$ give the two-sided linear estimates,
            and invariance of weighted area gives pairwise noncongruence.
    \end{proof}

        By \eqref{inmt}, fixed closed minimal inputs satisfying \eqref{pmix}
            yield pairwise distinct induced volumes for fixed $N$, with two-sided linear growth in $N$.

\subsection{Weighted doublings}\label{dblsec}

        The linearized doubling method joins two nearby sheets of a two-sided minimal surface
            by small catenoidal bridges and perturbs the result to an embedded minimal surface;
        see \cite{KY10,KM19,Wiy20} for spherical examples.
       
        On a finite angular quotient of a tube contained in $\mathcal S^\circ$,
            we will verify the following hypotheses of the Kapouleas--McGrath theorem
            \cite[Convention 2.1 and Assumption 7.2]{KM23}:
    \begin{enumerate}
    \item[(KM1)] a connected closed orientable embedded two-sided minimal torus $T\subset(N^3,g)$,
    with unit normal $\nu$;
    \item[(KM2)] effective isometric $O(2)\times\Z_2$ actions on $T$ and $N$
        for which the inclusion $T\hookrightarrow N$ is equivariant,
    with the $\Z_2$ generator reversing the orientation of $T$;
    \item[(KM3)] positive Jacobi potential $|\mathrm A_T|^2+\operatorname{Ric}(\nu,\nu)$,
        where $\mathrm A_T$ is the second fundamental form;
    \item[(KM4)] the scalar Jacobi operator
    $\mathcal J_T=\Delta_T+|\mathrm A_T|^2+\operatorname{Ric}(\nu,\nu)$ of $T\subset(N,g)$
    has trivial kernel on the $O(2)\times\Z_2$-invariant functions.
    \end{enumerate}
        Under (KM1)--(KM4), the theorems of Kapouleas--McGrath
            \cite[Theorems 9.39--9.40]{KM23}
            produce connected embedded doublings converging as varifolds to $2|T|$.

        On $\mathcal S^\circ$, where $\mathcal S=S(\C\oplus\C)=\Sphere^3$,
            let $W:=W_{k_1,k_2}=|z_1|^{k_1}|z_2|^{k_2}$ and $g':=Wg_{\Sphere^3}$,
        and use $(z_1,z_2)=(\cos s\,e^{\mathbf i\theta},\sin s\,e^{\mathbf i\varphi})$.
        The $g'$-area of the angular orbit $T_s$ is proportional to
            $\cos^{k_1+1}s\sin^{k_2+1}s$, with unique critical orbit $T_*=\{s=s_*\}$,
        where $\cos^2s_*=(k_1+1)/(k_1+k_2+2)$.
        The angular $\T^2$-action is isometric and of cohomogeneity one, so $T_*$ is minimal.
        Set $W_*=W|_{T_*}$.
        For the induced metric and scalar Jacobi operator of $T_*\subset(\mathcal S^\circ,g')$,
        we have
    \begin{equation}\label{pjac}
    \begin{aligned}
        &g_{T_*}=\frac{W_*}{k_1+k_2+2}
        \bigl((k_1+1)d\theta^2+(k_2+1)d\varphi^2\bigr),
        \\
        &|\mathrm A_{T_*}|^2+\operatorname{Ric}(\nu,\nu)
        =\frac{2(k_1+k_2+2)}{W_*}>0,
        \\
        &\mathcal J_*=\frac{k_1+k_2+2}{W_*}
        \left(\frac{\partial_{\theta\theta}}{k_1+1}
        +\frac{\partial_{\varphi\varphi}}{k_2+1}+2\right).
    \end{aligned}
    \end{equation}
        Using the fact that $\partial_s$ has $g'$-normal speed $W_*^{1/2}$,
            together with the second variation along parallel orbits and the logarithmic area derivative
            $-2(k_1+k_2+2)$ at $s_*$, we obtain the middle identity.
        The last line is the induced Laplacian plus this potential.

    \begin{theorem}[Weighted doubling profiles]\label{pdbl}
        Every sufficiently small angularly invariant tube $U$ about $T_*$ contains a sequence
            of connected closed embedded $g'$-minimal (equivalently, $W$-minimal) surfaces $\Sigma_j$.
        They are invariant under both coordinate sign changes, their genera tend to infinity,
        and they converge as varifolds to $2|T_*|$, with $\AW_W(\Sigma_j)\to 2\AW_W(T_*)$.
    \end{theorem}

    \begin{proof}
        Let $\mathcal R_n$ denote the $n$th roots of unity.
        Fix such a tube $U\Subset\mathcal S^\circ$,
        choose an even integer $q_0$ with $q_0^2>2(k_2+1)$, and set $G_0=\{\pm 1\}\times\mathcal R_{q_0}$.
        The free angular quotient $\bar T_*=T_*/G_0\subset\bar U=U/G_0$
            is a connected closed orientable embedded two-sided minimal torus,
        verifying (KM1).
        In the $2\pi$-periodic coordinates $\Theta=2\theta$ and $\Phi=q_0\varphi$,
        the standard $O(2)$-action on $\Theta$ and the reflection $\Phi\mapsto-\Phi$ verify (KM2).
        The positive Jacobi potential in \eqref{pjac} verifies (KM3).

        The quotient is locally isometric for $g'$ and preserves the normal $W_*^{-1/2}\partial_s$,
            so the Jacobi operator of $\bar T_*$ is \eqref{pjac}
            with $\partial_\theta=2\partial_\Theta$ and $\partial_\varphi=q_0\partial_\Phi$.
        The invariant normal speeds are independent of $\Theta$ and even in $\Phi$.
        Thus the Fourier modes $\cos(n\Phi)$, $n\geq0$, have Jacobi eigenvalues
            $(k_1+k_2+2)W_*^{-1}\bigl(2-q_0^2n^2/(k_2+1)\bigr)$.
        Our choice of $q_0$ makes the constant eigenvalue positive and all nonconstant eigenvalues negative,
        verifying (KM4).

        Choose one of the admissible finite configurations of singular parallel circles
            in \cite[Theorems 9.39--9.40]{KM23},
        and let the number $M$ of bridges on each circle tend to infinity.
        By these theorems, we obtain connected closed embedded minimal doublings $\bar\Sigma_M\subset\bar U$
            with genera tending to infinity and varifold limit $2|\bar T_*|$.

        Lift these doublings through the fixed angular covering $U\to\bar U$.
        The lifts $\Sigma_M$ are connected because a graphical sheet over the punctured quotient torus
            has surjective monodromy;
        the covering orientation is $G_0$-invariant,
            and evenness of $q_0$ gives both coordinate sign symmetries.
        Genus growth and the varifold and area convergence
            \cite[Theorems 5.7 and 9.40, (5.9)]{KM23} pass through this fixed cover,
        so $\AW_W(\Sigma_M)\to 2\AW_W(T_*)$.
    \end{proof}

\subsection{Circle saturation of doublings}\label{satsec}

        We return to $\mathcal S=S(\HH\oplus\HH)=\Sphere^7$ and recall the horizontal slice $L$ in \eqref{hslice}.
        We first construct a minimal surface in $\bar L$ and then restore its orbit circles.
        The double cover $\operatorname{pr}|_L:L\to\bar L$ has deck group $H_{p,q}=\langle((-1)^p,(-1)^q)\rangle$.
        For a closed connected embedded $H_{p,q}$-invariant surface $\Sigma\subset L$,
            let $\Gamma=\Sigma/H_{p,q}$ and define its circle saturation by
    \begin{equation}\label{satmap}
        \widehat\Sigma
        =\operatorname{pr}^{-1}(\Gamma)
        =\bigl\{(ae^{\mathbf i p\vartheta},be^{\mathbf i q\vartheta}):
          (a,b)\in\Sigma,\ \vartheta\in\R/(2\pi\Z)\bigr\}
        \subset\mathcal S^\circ.
    \end{equation}
        This is an embedded principal $\Sphere^1$-bundle over $\Gamma$.

        For $r=2$, the reduction formula \eqref{hlred} induces on $L$
            the metric $\widetilde g_{p,q}=A(s)g_{\Sphere^3}$ up to the constant factor $2\pi$.
        Lemma \ref{hfix} and Hsiang--Lawson reduction therefore make $\widehat\Sigma$ $W_0$-minimal
            whenever $\Sigma$ is $\widetilde g_{p,q}$-minimal.
        If $\Gamma$ is orientable, the circle bundle is trivial:
            its pullback to $\Sigma$ has the tautological section,
        and pullback by this double cover is injective on $H^2(\,\cdot\,;\Z)\cong\Z$.
        Fiber integration gives
    \begin{equation}\label{satvol}
        \widehat\Sigma
        \cong\Sphere^1\times\Gamma,
        \qquad
        \AW_{W_0}(\widehat\Sigma)
        =\pi\operatorname{Area}_{\widetilde g_{p,q}}(\Sigma).
    \end{equation}
        We now supply such surfaces by the doubling construction of \S \ref{dblsec}.

    \begin{corollary}[Saturated doublings]\label{satdbl}
        For every pair of nonzero coprime integers $p,q$,
            $\mathcal S^\circ$ contains connected closed embedded $W_0$-minimal three-folds
            $\widehat\Sigma_j^{p,q}\cong\Sphere^1\times\Gamma_j$,
        where the $\Gamma_j$ are closed orientable surfaces with genera tending to infinity.
        The three-folds have uniformly bounded weighted volumes.
    \end{corollary}

    \begin{proof}
        Set $F(s)=A(s)\cos s\sin s$ and $x=\sin^2s$.
        Since $\log F^2=(k_1+1)\log(1-x)+(k_2+1)\log x
            +\log(p^2(1-x)+q^2x)$ is strictly concave,
        $F$ has a unique nondegenerate maximum at $s_{p,q}$.
        Thus $T_{p,q}=\{s=s_{p,q}\}$ is $\widetilde g_{p,q}$-minimal,
        with positive Jacobi potential.
        Set $\kappa=-(\log F)''(s_{p,q})>0$.
        On the angular quotient used in Theorem \ref{pdbl},
            the same $O(2)\times\Z_2$-invariant modes $\cos(n\Phi)$
            have Jacobi eigenvalues with the signs of $\kappa-q_0^2n^2/\sin^2s_{p,q}$.
        Choose the even integer $q_0$ sufficiently large to verify (KM4),
            while (KM1)--(KM3) hold as before.
        The Kapouleas--McGrath theorem on the quotient and the same argument with finite covers produce
            $\widetilde g_{p,q}$-minimal doublings $\Sigma_j$ of uniformly bounded area,
        whose quotients $\Gamma_j=\Sigma_j/H_{p,q}$ are orientable with genera tending to infinity.
        The remaining conclusions follow from \eqref{satmap}--\eqref{satvol}.
    \end{proof}

\section{Gaussian transfer for large exponents}\label{gaus}

        \S \ref{part} kept the weight fixed.
        We now consider the monomial weights and their total exponents
    \begin{equation}\label{gweight}
        W_j(r)=\prod_{i=1}^m r_i^{\delta_i^{(j)}},
        \qquad D_j:=\sum_{i=1}^m \delta_i^{(j)}.
    \end{equation}
        As $D_j\to\infty$ with positive limiting exponent proportions,
            rescaling about the maximizer at scale $D_j^{-1/2}$ gives Euclidean space
            with density $e^{-|y|^2}$.
        Rotationally nondegenerate Gaussian minimal submanifolds then give $W_j$-minimal profiles,
            with their number bounded below by the LS category of the rotation orbit.

\subsection{Gaussian blow-up at the balanced point}

        Let $B=\Sphere^{m-1}_{++}$ and assume the weights in \eqref{gweight} satisfy
    \begin{equation}\label{gpos}
        \delta_i^{(j)}>0,
        \qquad
        D_j\longrightarrow\infty,
        \qquad
        \frac{\delta_i^{(j)}}{D_j}
        \longrightarrow\lambda_i>0.
    \end{equation}
        Set $N=m-1$.
        The unique maximizer of $W_j$ is
            $p_j=(\sqrt{\delta_i^{(j)}/D_j})_i\to p_\infty=(\sqrt{\lambda_i})_i$,
        and $\operatorname{Hess}_{g_B}\log W_j|_{p_j}=-2D_jg_B$.
        This fixes the center and scale of the blow-up.
        Fix an orthonormal identification $I_\infty:\R^N\to T_{p_\infty}B$
            and transport it along the short geodesics to $p_j$ to obtain $I_j$.
        On fixed balls set
    \begin{equation}\label{gsca}
        \Phi_j(y)
        =\exp_{p_j}\left(\frac{I_jy}{\sqrt{D_j}}\right),
        \qquad
        \widehat g_j
        =D_j\Phi_j^*g_B,
        \qquad
        w_j
        =\frac{W_j\circ\Phi_j}{W_j(p_j)}.
    \end{equation}
        Other orthonormal identifications only rotate these coordinates, leaving the isotropic limit unchanged.

    \begin{proposition}[Gaussian blow-up]\label{glim}
        On every fixed ball $B_R(0)\subset\R^N$ and for every integer $h\geq 0$, as $j\to\infty$,
    \begin{equation}\label{gden}
        \log w_j(y)=-|y|^2+O_{C^h}(D_j^{-1/2}),\qquad
        \widehat g_j=g_{\R^N}+O_{C^h}(D_j^{-1}).
    \end{equation}
    \end{proposition}

    \begin{proof}
        Since $p_j\to p_\infty\in B$,
            every fixed higher covariant derivative of $\log W_j$ is $O(D_j)$ on a common neighborhood.
        Taylor expansion at scale $D_j^{-1/2}$ proves the density estimate in \eqref{gden}.
        The metric estimate follows from normal coordinates.
    \end{proof}

\subsection{Persistence and the transfer theorem}

        The Gaussian weighted area and its critical point equation are
    \begin{equation}\label{gfun}
        \mathcal F_G(\Sigma)
        =\int_\Sigma e^{-|y|^2}\,d\Vol,
        \qquad
        \mathbf H_\Sigma
        =-2y^\perp.
    \end{equation}
        We call the critical submanifolds \emph{Gaussian minimal}.
        The normal linearization of $\mathbf H_\Sigma+2y^\perp$
            is self-adjoint with respect to Gaussian $L^2$ and,
        after rescaling \cite[Theorem 4.1 and Remark 4.1]{ALW14} to the normalization $\mathbf H=-2y^\perp$,
        equals
    \begin{equation}\label{gjac}
        J^G_\Sigma V
        =\Delta^\perp V-2\nabla^\perp_{y^\top}V
         +\sum_{a,b} \langle A_{ab},V\rangle A_{ab}
         +2V.
    \end{equation}
        Here $\Delta^\perp$ is the nonpositive normal connection Laplacian
            and $A_{ab}=A(e_a,e_b)$ for a local orthonormal tangent frame,
        all computed in the Euclidean metric.
        Our sign convention is
            $ D^2\mathcal F_G(V,V) = -\int_\Sigma\langle J^G_\Sigma V,V\rangle e^{-|y|^2}\,d\Vol $.
        Rotations preserve $\mathcal F_G$,
            so their normal components form
            $\mathscr K_\Sigma:=\{(Ay)^\perp:A\in\mathfrak{so}(N)\}\subset\ker J^G_\Sigma$.
        We call a closed embedded Gaussian minimal $\Sigma$ a \emph{Gaussian seed}
            when it is \emph{rotationally nondegenerate}, that is, when
    \begin{equation}\label{gker}
        \ker J^G_\Sigma=\mathscr K_\Sigma.
    \end{equation}
        We now state the persistence result used below.
        The kernel condition is transverse nondegeneracy of the compact critical orbit;
            compare \cite{Cha91,Whi91}.
        All normal projections, graphs, and H\"older norms in the statement use $g_0$.
        We use unreduced LS category, summed over connected components, so $\operatorname{cat}(\mathrm{point})=1$.

    \begin{lemma}[Persistence along compact symmetry orbits]\label{gmb}
        Fix $0<\alpha<1$.
        Let a compact Lie group $G$ act isometrically on $(\mathcal N,g_0)$
            and preserve a smooth positive density $w_0$.
        Let $\Sigma^q\subset\mathcal N$, $1\leq q<\dim\mathcal N$, be compact, embedded, without boundary,
        and $w_0$-minimal.
        Let $J_0$ denote the normal linearization at $\Sigma$
            of $\mathbf H_{(\cdot),g_0}-(\nabla^{g_0}\log w_0)^\perp$,
        identified as an operator on $N_{g_0}\Sigma$, and assume that
    \[
        \ker J_0=\mathscr K:=\{X_A^\perp|_\Sigma:A\in\mathfrak g\},
    \]
        where $X_A$ is the infinitesimal field of $A$.
        Set $H=\{a\in G:a\Sigma=\Sigma\}$ and fix a relatively compact
            $G$-invariant neighborhood $U$ of $G\Sigma$.
        For every smooth metric $g$ and smooth positive density $w$ near $\overline U$, with
            $ \varepsilon
              :=\|g-g_0\|_{C^{3,\alpha}(\overline U)}
                +\|w-w_0\|_{C^{3,\alpha}(\overline U)} $
        sufficiently small,
        there are at least $\operatorname{cat}(G/H)$ pairwise distinct
            closed embedded $(g,w)$-weighted minimal submanifolds.
        Each is $a\operatorname{graph}^{g_0}_\Sigma u$ for some $a\in G$ and normal field $u$,
            with $\|u\|_{C^{2,\alpha}}=O(\varepsilon)$.
        The perturbed data need not be $G$-invariant.
        If $\ker J_0=0$, the solution in a fixed small normal graph neighborhood of $\Sigma$ is unique
            and depends smoothly on jointly smooth finite-dimensional families of $(g,w)$.
    \end{lemma}

        The proof in Appendix \ref{persist} combines transverse reduction by the implicit function theorem
            with LS theory on $G/H$.

        For a Gaussian seed, take $G=O(N)$, $g_0=g_{\R^N}$, and $w_0=e^{-|y|^2}$.
        Then $J_0=J^G_\Sigma$ in \eqref{gjac}, and the kernel hypothesis is precisely \eqref{gker}.

    \Needspace{9\baselineskip}
    \begin{theorem}[Gaussian transfer with an orbit category bound]\label{gtra}
        Assume \eqref{gpos}, fix $0<\alpha<1$,
            and let $\Sigma^q\subset\R^N$, $1\leq q<N$,
            be a Gaussian seed with stabilizer $H\leq O(N)$.
        For all sufficiently large $j$,
            there are at least $\operatorname{cat}(O(N)/H)$ pairwise distinct
            closed embedded $W_j$-minimal profiles $Q_{j,\ell}\cong\Sigma$,
        all contained in a common compact subset of $B$.
        Up to rotations, their rescaled preimages $\Phi_j^{-1}(Q_{j,\ell})$
            converge smoothly to $\Sigma$ as normal graphs
            of $C^{2,\alpha}$-size $O(D_j^{-1/2})$,
        uniformly in $\ell$.
        They also satisfy
    \begin{equation}\label{gvolgen}
        \frac{D_j^{q/2}}{W_j(p_j)}
        \int_{Q_{j,\ell}}W_j\,d\Vol_{g_B}
        \longrightarrow
         \int_\Sigma e^{-|y|^2}\,d\Vol.
    \end{equation}
    \end{theorem}

    \begin{proof}
        Fix a ball containing the rotation orbit of $\Sigma$ and all its sufficiently small normal graphs.
        Applying Proposition \ref{glim} and Lemma \ref{gmb} to $(\widehat g_j,w_j)$,
            with $\varepsilon_j=O(D_j^{-1/2})$, we obtain profiles $\widehat Q_{j,\ell}$ there
        and set $Q_{j,\ell}=\Phi_j(\widehat Q_{j,\ell})$.
        The exact identity
    \[
        \frac{D_j^{q/2}}{W_j(p_j)}
        \int_{Q_{j,\ell}}W_j\,d\Vol_{g_B}
        =\int_{\widehat Q_{j,\ell}}w_j\,d\Vol_{\widehat g_j}
    \]
        holds for every nearby variation, so criticality is preserved under $\Phi_j$.
        The graph estimate and Proposition \ref{glim} imply \eqref{gvolgen}.
        Elliptic bootstrapping gives smooth convergence.
        Since $p_j\to p_\infty\in B$ and the rescaled graphs remain bounded,
            all profiles lie in a common compact subset of $B$.
    \end{proof}

    \compactnextenvtrue
    \begin{shortremark}[Anisotropic Gaussian limits]
        In more general concentrating geometries the negative rescaled Hessian need not be scalar,
            leading instead to the density $e^{-\langle\mathsf Ay,y\rangle/2}$
            for a positive definite symmetric matrix $\mathsf A$.
        Lemma \ref{gmb} then applies with
            $G_{\mathsf A}=\{R\in O(N):R^T\mathsf AR=\mathsf A\}$
            once convergence and the kernel condition are verified;
        when the full kernel is zero, one takes $G=\{1\}$.
    \end{shortremark}

\subsection{Spherical and projective seeds}

        We first treat round spheres by a spectral calculation.
        For Veronese projective planes,
            only the Jacobi fields normal to the ambient sphere require a separate argument.

    \begin{lemma}[Round Gaussian seeds]\label{gseed}
        Let $1\leq q<N$.
        For every $(q+1)$-plane $E\subset\R^N$,
            the round sphere $\Sigma_{q,E}=\Sphere^q(\sqrt{q/2})\subset E$ is a Gaussian seed.
    \end{lemma}

    \begin{proof}
        At radius $\sqrt{q/2}$, the sphere satisfies $\mathbf H=-2y$.
        Split the Euclidean normal bundle into the radial line and the constant bundle $E^\perp$.
        On these summands, $J^G_\Sigma$ is $\Delta_\Sigma+4$ and $\Delta_\Sigma+2$, respectively.
        Since the eigenvalues of $-\Delta_\Sigma$ are $2\ell(\ell+q-1)/q$, the radial operator has no kernel,
        while the transverse kernel consists of the first harmonics.
        These are exactly the normal components of rotations mixing $E$ and $E^\perp$.
    \end{proof}

    \begin{corollary}[Interior spherical profiles and their compositions]\label{gint}
        Assume $m\geq 3$ and \eqref{gpos}.
        For all sufficiently large $j$ and every $1\leq q\leq m-2$,
        there exist at least $\operatorname{cat}(\operatorname{Gr}_{q+1}(\R^{m-1}))$
            pairwise distinct closed embedded $W_j$-minimal profiles
            $Q_{j,q,\ell}\cong\Sphere^q$
            satisfying the convergence and compactness conclusions of Theorem \ref{gtra}.

        For positive integer exponents $\delta_i^{(j)}=k_i^{(j)}$,
            real scalar composition with the identity embeddings of the spheres produces the same number
            of pairwise distinct closed minimal embeddings of codimension $m-1-q$:
    \begin{equation}\label{grealm}
        F_{j,q,\ell}:\Sphere^q\times\prod_{i=1}^m \Sphere^{k_i^{(j)}}
        \longrightarrow \Sphere^{D_j+m-1}.
    \end{equation}
        For fixed $m,q$ and every $\varepsilon>0$,
            the threshold in $D=\sum_i k_i$ is uniform over positive integer tuples
            with $\min_i k_i/D\geq \varepsilon$.
        The embeddings are pairwise noncongruent when their input dimensions are pairwise distinct.
        When $q=m-2$, every block norm is nonconstant.
    \end{corollary}

    \compactnextenvtrue
    \begin{proof}
        Set $N=m-1$.
        The round seed has stabilizer $O(q+1)\times O(N-q-1)$,
            so its orbit is $\operatorname{Gr}_{q+1}(\R^N)$.
        Lemma \ref{gseed} and Theorem \ref{gtra} prove the first assertion,
            with a common threshold for the finitely many $q$.
        When $q=m-2$, a constant block norm $r_i=c$ would make the closed profile
            both open and closed in $\{r_i=c\}\cap B$,
            by equality of dimensions and compactness.
        This level set is connected and noncompact, a contradiction.

        Theorem \ref{comp} makes the real scalar compositions minimal.
        Block norms and normalized blocks recover the profile and sphere factors, respectively,
            proving embeddedness and preserving distinct profile images.
        Proposition \ref{congr} then shows that they are noncongruent when the input dimensions are distinct.
        If the uniform threshold failed, a counterexample sequence would have a subsequence
            with $k_i^{(j)}/D_j\to\lambda_i\geq \varepsilon$,
            contradicting the preceding conclusions.
    \end{proof}

    \begin{proposition}[Criterion for spherical seeds]\label{sphcrit}
        Let $X:M^q\hookrightarrow\Sphere^n(1)\subset E=\R^{n+1}$, $q\geq 1$, be a full closed minimal embedding,
        let $E\subset\R^N$ with $N\geq n+1$, and set $\Sigma=\sqrt{q/2}\,X(M)\subset\R^N$.
        Suppose that
        \begin{enumerate}
        \item every spherical normal Jacobi field of $X$ is the normal component
            of an infinitesimal rotation of $E$;
        \item $2q\notin\operatorname{Spec}(-\Delta_X)$;
        \item if $N>n+1$, the $q$-eigenspace of $-\Delta_X$ is the span of the coordinate functions of $X$.
        \end{enumerate}
        Then $\Sigma$ is a Gaussian seed.
    \end{proposition}

    \begin{proof}
        Minimality gives $\mathbf H_X^E=-qX$, hence $\mathbf H_\Sigma=-2y$.
        Under dilation,
            $\mathcal N_{\R^N}\Sigma=\R y\oplus N^{\Sphere}X\oplus E^\perp$
            and $\Delta_\Sigma=(2/q)\Delta_X$.
        With respect to this splitting,
    \[
        J^G_\Sigma
        =(\Delta_\Sigma+4)\oplus\frac{2}{q}J_X^{\Sphere}
         \oplus(\Delta_\Sigma+2).
    \]
        Here $J_X^{\Sphere}$ is the spherical normal Jacobi operator.
        Assumption (2) removes the radial kernel,
        (1) identifies the spherical normal kernel with rotations within $E$,
            and (3) identifies the remaining kernel with rotations mixing $E$ and $E^\perp$.
    \end{proof}

    \begin{proposition}[Projective Gaussian seeds]\label{prjseed}
        Let $\mathbb K\in\{\R,\C,\mathbb H,\mathbb O\}$, and set $a=\dim_{\R}\mathbb K$.
        Let $X_a:\mathbb KP^2\hookrightarrow\Sphere^{3a+1}(1)$ be the standard Veronese embedding.
        For $N\geq 3a+2$, $\Sigma_a=\sqrt{a}\,X_a(\mathbb KP^2)\subset\R^N$ is a Gaussian seed.
    \end{proposition}

    \compactnextenvtrue
    \begin{proof}
        Set $q=2a$ and $N_0=3a+2$.
        The Veronese image is the full closed minimal $q$-dimensional focal submanifold
            of Cartan's $g=3$ family \cite[Section 2]{TZ20}.
        We verify the hypotheses of Proposition \ref{sphcrit}.
        Let $g_{\mathrm{std}}$ be the diameter-$\pi/2$ canonical projective metric
            used in \cite{ADGMS23}.
        In the projector model, $X_a^*g_{\Sphere^{3a+1}}=3g_{\mathrm{std}}$.
        After this rescaling, \cite[Theorem 2.3 and Table 2]{ADGMS23} shows that
            $\nu_\ell=\frac{4}{3}\ell(\ell+\frac{3a}{2}-1)$,
        $\ell\geq 0$.
        The first positive eigenvalue is $\nu_1=q=2a$, with multiplicity $3a+2=N_0$.
        Since the immersion is linearly full and $-\Delta_{X_a} X_a=q X_a$,
            this eigenspace is exactly the coordinate span.
        Moreover, $\nu_2>2q$, so $2q$ is absent from the spectrum.

        It remains to identify the kernel on the spherical normal summand.
        For $a=2,4,8$, \cite[Theorem 1]{RS26} shows that the normal nullity equals the Killing nullity,
        $\operatorname{Nul}(X_a)=\operatorname{Nul}_K(X_a)$,
            so every spherical normal Jacobi field is induced by an ambient rotation.
        For $a=1$, pull a spherical normal Jacobi field $V$ back by $\pi:\Sphere^2\to\mathbb{RP}^2$.
        In the notation of \cite{EM08}, $\deg(\xi^{0,1}\otimes_{\C}\kappa)=-4$,
            so the obstruction space vanishes.
        By \cite[(2.9)]{EM08}, we can therefore choose a tangential correction $\sigma$
            making $\pi^*V+\sigma$ infinitesimally conformal.
        The polarized comparison formula \cite[(2.27)]{EM08}, together with the Jacobi equation for area,
        then makes it a Jacobi field for harmonic maps.
        Averaging $\sigma$ over the deck group makes this field descend to $\mathbb{RP}^2$.
        By \cite[Example 4.3]{KNPS25}, it is $BX_1$ for some $B\in\mathfrak{so}(5)$,
            so $V=(BX_1)^\perp$.

        Proposition \ref{sphcrit} proves the result for every $N\geq N_0$.
    \end{proof}

        Under \eqref{gpos} and $m\geq3a+3$, Theorem \ref{gtra} produces
            $\mathbb KP^2$-profiles with the stated convergence and orbit category bounds.
        For integer exponents $\delta_i^{(j)}=k_i^{(j)}$, Example \ref{normcmp} gives closed minimal embeddings of
            $\mathbb KP^2\times\prod_i\Sphere^{k_i^{(j)}}$ into $\Sphere^{D_j+m-1}$,
            of codimension $m-1-2a$.
        They are pairwise noncongruent when the input dimensions are distinct, by Proposition \ref{congr}.

\subsection{Orbit categories and profile counts}\label{orbcnt}

        For the rotation orbit $\mathcal O_\Sigma=O(N)/H$,
            we estimate $\operatorname{cat}(\mathcal O_\Sigma)$ using the standard bounds
    \[
        \operatorname{cuplength}_{\mathbb F}(X)+1
        \leq \operatorname{cat}(X),
        \qquad
        \operatorname{cat}(X)
        \leq \left\lfloor\frac{\dim X}{r}\right\rfloor+1,
    \]
        for a connected finite CW complex $X$, the latter when $X$ is $(r-1)$-connected
            \cite[Proposition 1.5 and Theorem 1.50]{CLOT03}.

    \begin{table}[htbp]
        \centering
        \small
        \renewcommand{\arraystretch}{1.2}
        \begin{tabular}{@{}llll@{}}
        \hline
        Gaussian seed & Ambient space & Rotation orbit $\mathcal O_\Sigma$
        & $\operatorname{cat}(\mathcal O_\Sigma)$
        \\
        \hline
        $\Sphere^{N-1}$ & $\R^N$ & point & $1$
        \\
        $\Sphere^{N-2}$, $N\geq 3$ & $\R^N$ & $\mathbb{RP}^{N-1}$ & $N$
        \\
        $\Sphere^q$, $1\leq q\leq N-3$ & $\R^N$
        & $\operatorname{Gr}_{q+1}(\R^N)$ & $\geq 2$
        \\
        $X_1(\mathbb{RP}^2)$ & $\R^5$ & $O(5)/SO(3)$ & $6$
        \\
        $\sqrt{2}\,X_2(\mathbb{CP}^2)$ & $\R^8$ & $SO(8)/PU(3)$ & $\geq 4$
        \\
        $2X_4(\mathbb{HP}^2)$ & $\R^{14}$ & $O(14)/PSp(3)$ & $\geq 10$
        \\
        $\sqrt{8}\,X_8(\mathbb{OP}^2)$ & $\R^{26}$ & $O(26)/F_4$ & $\geq 20$
        \\
        \hline
        \end{tabular}
        \caption{Orbit categories.
        The projective rows use their minimal ambient spans $N=3a+2$,
        and the spherical radii are $\sqrt{q/2}$.
        Each entry is a lower bound for the number of profiles for large exponents;
        inequalities need not be equalities.}
        \label{seedcat}
    \end{table}

        A round $q$-sphere has orbit $\operatorname{Gr}_{q+1}(\R^N)$:
            the two extreme cases have categories $1$ and $N$,
        and every other case has category at least $2$ by noncontractibility.

        The standard Veronese images span their trace-free Jordan modules,
            so their full orthogonal stabilizers in the minimal spans are
    \[
        H_1=SO(3),\qquad H_2=PU(3)\rtimes\Z_2,
        \qquad H_4=PSp(3),\qquad H_8=F_4.
    \]
        These are the effective isometry groups in the standard representations \cite[Section 3]{AB03}.
        The extra involution in $H_2$ reverses orientation on the trace-free Hermitian module.
        Hence $O(8)/H_2\cong SO(8)/PU(3)$ is connected,
            while the other three full $O(N)$-orbits have two isomorphic components.

        Each real component is the $2$-connected Berger space $B^7=SO(5)/SO(3)$,
            an $S^3$-bundle over $S^4$ with Euler class $\pm 10$ \cite[Corollary 2]{GKS04}.
        Since this Euler class is even,
            the mod-$2$ Serre spectral sequence shows that $B^7$ has cup length $2$,
        while the connectivity bound gives $\operatorname{cat}(B^7)\leq 3$.
        Thus $\operatorname{cat}(B^7)=3$, and the two components contribute $6$.

        We compute the nonreal rows using Borel's Cartan model \cite{Bor53}.
        For $PU(3)$, eliminating the universal Chern class $c_2$ leaves
            $\mathbb Q[c_3]/(c_3^2)\otimes\Lambda(x_7,y_7)$,
            where $\deg c_3=6$; its cup length is $3$.
        For $(N,H)=(14,PSp(3))$ or $(26,F_4)$, the restriction
            $H^*(BSO(N);\mathbb Q)\to H^*(BH;\mathbb Q)$ is surjective:
        the Pontryagin classes $p_1,p_2,p_3$ generate in the first case,
            and $p_1,p_3,p_4,p_6$ generate in the second,
            by \cite[Theorems 4.2.4 and 5.1.1]{JRT26} and Newton identities.
        Cancelling the corresponding pairs in the Cartan model leaves exterior algebras on
            $\operatorname{rank}SO(N)-\operatorname{rank}H=4,9$ odd generators, respectively.
        The cup length bound and the component counts above give
            $4$, $10$, and $20$ in Table \ref{seedcat}.

        If $N>3a+2$, the orbit is $O(N)/(H_a\times O(N-3a-2))$ and Theorem \ref{gtra} still applies.

\subsection{Two applications of Gaussian transfer}

        We couple one curve obtained by Gaussian transfer with two types of inner minimal inputs.

    \compactnextenvtrue
    \begin{proposition}[Coupled families]\label{coupled}
        There is an integer $n_0$ such that, for every integer $n\geq n_0$,
            the round sphere $\Sphere^{2n-3}$ contains closed embedded
            $n$-dimensional minimal submanifolds $Y_{\mathbf j}^n$,
        indexed by $\mathbf j=(j_1,j_2,j_3)$, whose volumes are bounded uniformly in $\mathbf j$.
        Their three underlying surface genera vary independently through unbounded sequences,
            and each of $b_1(Y_{\mathbf j}^n),\ldots,b_6(Y_{\mathbf j}^n)$
            diverges with any one index while the other two remain fixed.
        No nonempty open subset of any $Y_{\mathbf j}^n$ is a nontrivial Riemannian product.
    \end{proposition}

    \begin{proof}
        Choose $n_a=k_a+\ell_a+2$, $k_a,\ell_a\geq2$, so that
            $\sum_{a=1}^3n_a=n-1$ and $|n_a-n_b|\leq1$.
        For large $n$, $n_a/(n-1)\geq1/4$; Corollary \ref{gint}, with $m=3$ and $q=1$,
            therefore gives a closed embedded $V$-minimal curve
            $\gamma=(\rho_1,\rho_2,\rho_3)\subset\Sphere^2_{++}$,
            where $V(r)=\prod_a r_a^{n_a}$.
        Parametrize it by round arclength $t\in\R/L\Z$.
        Its coordinates are nonconstant, and its weighted geodesic equation makes them real analytic.

        For $w_a(z)=|z_1|^{k_a}|z_2|^{\ell_a}$, choose closed embedded $w_a$-minimal surfaces
            $\Sigma_{a,j_a}\subset S(\C\oplus\C)^\circ$
            invariant under the coordinate sign changes, as supplied by Theorem \ref{pdbl}.
        Proposition \ref{etens} gives minimal embeddings
    \[
        f_{a,j_a}:M_{a,j_a}
        =\bigl(\Sigma_{a,j_a}\times\Sphere^{k_a}\times\Sphere^{\ell_a}\bigr)/(\Z_2)^2
        \hookrightarrow\Sphere^{2n_a-1}.
    \]
        Example \ref{normcmp} then gives the minimal embedding
    \[
        F_{\mathbf j}:(\R/L\Z)\times\prod_a M_{a,j_a}\longrightarrow\Sphere^{2n-3},
        \qquad F_{\mathbf j}(t,x)=(\rho_a(t)f_{a,j_a}(x_a))_{a=1}^3,
    \]
        whose image is $Y_{\mathbf j}^n$.

        The projection $M_{a,j_a}\to B_{a,j_a}:=\Sigma_{a,j_a}/(\Z_2)^2$
            has the section $[z]\mapsto[z,z_1/|z_1|,z_2/|z_2|]$,
            using fixed great circles in the sphere factors;
        simultaneous sign changes leave its class unchanged.
        Thus $H^*(B_{a,j_a};\mathbb Q)$ injects into $H^*(M_{a,j_a};\mathbb Q)$.
        The bases are orientable finite covers of the quotient doublings in Theorem \ref{pdbl},
            with genera $g^{\mathrm{top}}_{a,j_a}\to\infty$ independently.
        K\"unneth gives
    \[
        b_\ell(Y_{\mathbf j}^n)
        \geq[u^\ell](1+u)\prod_a(1+2g^{\mathrm{top}}_{a,j_a}u+u^2),
        \qquad 0\leq\ell\leq7.
    \]
        For $1\leq\ell\leq6$, the coefficient of each $g^{\mathrm{top}}_{a,j_a}$ is at least
            $2[u^{\ell-1}](1+u)(1+u^2)^2>0$, proving the Betti assertions.
        Theorem \ref{comp} also gives
            $\Vol(Y_{\mathbf j}^n)=C_n\prod_a\AW_{w_a}(\Sigma_{a,j_a})$,
            with $C_n>0$ independent of $\mathbf j$;
        Theorem \ref{pdbl} therefore bounds these volumes uniformly.

        Finally, Theorem \ref{comp} gives
            $g_{Y_{\mathbf j}^n}=dt^2+\sum_a\rho_a(t)^2h_{a,j_a}$,
            where $h_{a,j_a}=f_{a,j_a}^*g_{\Sphere^{2n_a-1}}$.
        Since each $\rho_a$ is nonconstant, periodic, and analytic,
            all $\rho_a''$ are nonzero outside finitely many points of $\R/L\Z$.
        For $T=\partial_t$ and $X=\sum_aX_a\perp T$, the warped curvature formula \cite{DU05} gives
    \[
        R(X,T)T=-\sum_a\frac{\rho_a''}{\rho_a}X_a,
    \]
        so $R(X,T)T=0$ forces $X=0$ there.
        In a nontrivial local product, however, such a nonzero $X$ always exists:
        if $T=T_1+T_2$ has two nonzero factor components, take
            $X=|T_2|^2T_1-|T_1|^2T_2$; otherwise take $X$ in the other factor.
        Every nonempty open set meets these points, proving local irreducibility.
    \end{proof}

    \begin{corollary}[Ordinary minimal inputs]
    \label{codimseven}
        The conclusions of Proposition \ref{coupled} also hold for a family
            $Z_{\mathbf j}^n\subset\Sphere^{n+7}$.
    \end{corollary}

    \begin{proof}
        Use the same $n_a$ and outer curve $\gamma$.
        Choose closed connected orientable embedded minimal surfaces
            $u_j:\Sigma_j\hookrightarrow\Sphere^3$
            with uniformly bounded areas and genera tending to infinity \cite{KY10,KM23}.
        Their classical minimal products
    \[
        f_{a,j}(x,y)=\left(\sqrt{2/n_a}\,u_j(x),\sqrt{1-2/n_a}\,y\right),
        \qquad(x,y)\in\Sigma_j\times\Sphere^{n_a-2},
    \]
        lie in $\Sphere^{n_a+2}$.
        Example \ref{normcmp} couples them with $\gamma$ to give
            $Z_{\mathbf j}^n\cong\Sphere^1\times\prod_a(\Sigma_{j_a}\times\Sphere^{n_a-2})$
            minimally embedded in $\Sphere^{n+7}$.
        K\"unneth gives the same Betti lower bound, and Theorem \ref{comp} bounds the volume by
            a fixed multiple of $\prod_a\operatorname{Area}(\Sigma_{j_a})$.
        The induced metric has the same warped form,
            so the curvature argument above proves local irreducibility.
    \end{proof}

{\ }

\appendix
\Needspace{8\baselineskip}
\section{Persistence along compact symmetry orbits}\label{persist}
    \begin{proof}[Proof of Lemma \ref{gmb}]
        \emph{The transverse family.}
        Use the fixed weighted pairing
            $\langle u,v\rangle_0
             =\int_\Sigma\langle u,v\rangle_{g_0}w_0\,d\Vol_{\Sigma,g_0}$.
        Let $P$ be the orthogonal projection onto $\mathscr K$ and set
            $\mathcal X^{k,\alpha}
            =\mathscr K^{\perp_0}\cap C^{k,\alpha}(N_{g_0}\Sigma)$.
        Write $\Gamma_u=\{\exp_y^{g_0}u(y):y\in\Sigma\}$ and
            $\mathcal A(a,u)=\int_{a\Gamma_u}w\,d\Vol_g$.
        Integration by parts in the first variation defines the
            $C^{0,\alpha}$ Euler--Lagrange operator $\mathcal E$ by
    \begin{equation}\label{gfirst}
        D_u\mathcal A(a,u)[v]
        =-\langle\mathcal E(a,u),v\rangle_0.
    \end{equation}
        For small graphs, their normal variations span all normal directions,
            so $\mathcal E=0$ is equivalent to weighted minimality.
        At $(g,w)=(g_0,w_0)$, invariance gives
            $\mathcal E(a,0)=0$ and $D_u\mathcal E(a,0)=J_0$.
        Weighted self-adjointness, ellipticity, and the kernel hypothesis make
            $J_0:\mathcal X^{2,\alpha}\to\mathcal X^{0,\alpha}$ an isomorphism.
        The implicit function theorem with parameters, uniformly over compact $G$,
            and elliptic regularity therefore give a unique small smooth family
            $u(a)\in\mathcal X^{2,\alpha}$ satisfying
    \[
        (I-P)\mathcal E(a,u(a))=0,
        \qquad \|u(a)\|_{2,\alpha}\leq C\varepsilon.
    \]
        For $h\in H$, define
            $(h\cdot u)(y)=d h_{h^{-1}y}u(h^{-1}y)$.
        Since $ah\Gamma_{h^{-1}\cdot u}=a\Gamma_u$ and the $H$-action
            preserves the fixed pairing and $\mathscr K$,
    \[
        \mathcal E(ah,h^{-1}\cdot u)
        =h^{-1}\cdot\mathcal E(a,u).
    \]
        This identity holds without $H$-invariance of $(g,w)$.
        Uniqueness gives $u(ah)=h^{-1}\cdot u(a)$;
            hence $b(a)=\mathcal A(a,u(a))$ descends to $G/H$.

        \medskip\noindent\emph{The orbit directions.}
        Choose an $\operatorname{Ad}(G)$-invariant inner product on $\mathfrak g$.
        The map
            $Z:\mathfrak h^\perp\to\mathscr K$,
            $Z_A=X_A^\perp|_\Sigma$,
        is an isomorphism, where $\mathfrak h=\operatorname{Lie}(H)$.
        Here $\exp(A)$ denotes the Lie group exponential.
        For small $u\in\mathcal X^{2,\alpha}$ and $A\in\mathfrak h^\perp$,
            tubular projection writes
            $\exp(A)\Gamma_u=\Gamma_{v_u(A)}$.
        Set $\Theta_u(A)=Pv_u(A)$.
        Changing variables by tubular projection expresses these fixed $L^2$ moments
            as integrals of smooth functions of $A,u,du$,
        so their first two $A$-derivatives are uniformly controlled
            for $C^{2,\alpha}$-small $u$.
        Since $v_u(0)=u$ and $D_Av_0(0)[A]=Z_A$,
            Taylor expansion in the finite-dimensional variable $A$ gives
    \begin{equation}\label{gorbit}
        \Theta_u(A)
        =Z_A+O_{\mathscr K}
            \bigl((\|u\|_{2,\alpha}+|A|)|A|\bigr).
    \end{equation}
        Thus $D\Theta_u(0)$ is invertible and
            $\Theta_u(A)=0$ implies $A=0$,
        uniformly for sufficiently small $u$ and $A$.

        At $u=u(a)$, we have $\mathcal E(a,u)\in\mathscr K$,
            while derivatives of $u(a)$ lie in $\mathcal X^{2,\alpha}$.
        Expressing the orbit variation as a normal graph in \eqref{gfirst} gives
    \begin{equation}\label{gred}
        \left.\frac{d}{dt}\right|_{t=0}b(a\exp(tA))
        =-\langle\mathcal E(a,u(a)),
                    D\Theta_{u(a)}(0)A\rangle_0.
    \end{equation}
        The invertibility of $D\Theta_{u(a)}(0)$ shows that
            every critical coset satisfies $\mathcal E(a,u(a))=0$.

        \medskip\noindent\emph{Distinct images and the count.}
        Suppose $a\Gamma_u=a'\Gamma_v$ for two members of this family.
        For sufficiently small $\varepsilon$, compactness puts
            $Q=a^{-1}a'$ in a small neighborhood of $H$:
        any limiting relative element as $u,v\to0$ must preserve $\Sigma$.
        Write $Q=\exp(A)h$ with $h\in H$ and $A\in\mathfrak h^\perp$ small.
        Since $H$ preserves $\mathcal X^{2,\alpha}$ and
            $\Gamma_u=\exp(A)\Gamma_{h\cdot v}$,
        we have $\Theta_{h\cdot v}(A)=Pu=0$.
        Equation \eqref{gorbit} gives $A=0$, hence $aH=a'H$.
        Distinct critical cosets therefore give distinct images,
            and Lusternik--Schnirelmann theory \cite{Tak68}
            supplies at least $\operatorname{cat}(G/H)$ of them.
        If $\mathscr K=0$, the projected equation is the full equation,
            so the same implicit function theorem gives uniqueness
            in the fixed small normal graph neighborhood
            and smooth dependence on jointly smooth finite-dimensional data parameters.
    \end{proof}


\begin{thebibliography}{99}
    \scriptsize

    \bibitem{AB03}
        M. F. Atiyah and J. Berndt, \emph{Projective planes, Severi varieties and spheres}, Surv.
        Differ. Geom. \textbf{8} (2003), 1--27.

    \bibitem{ADGMS23}
        A. Anderson, M. Dostert, P. J. Grabner, R. W. Matzke and T. A. Stepaniuk, \emph{Riesz and Green energy on projective spaces},
        Trans. Amer. Math. Soc. Ser. B \textbf{10} (2023), 1039--1076.

    \bibitem{ALW14}
        B. Andrews, H. Li and Y. Wei, \emph{$\mathcal F$-stability for self-shrinking solutions to mean curvature flow},
        Asian J. Math. \textbf{18} (2014), no. 5, 757--778.

    \bibitem{Bor53}
        A. Borel, \emph{Sur la cohomologie des espaces fibr\'es principaux et des espaces homog\`enes de groupes de Lie compacts},
        Ann. of Math. (2) \textbf{57} (1953), no. 1, 115--207.

    \bibitem{Cha91}
        K.-C. Chang, \emph{A remark on the perturbation of critical manifolds}, Acta Math.
        Sci. (English Ed.) \textbf{11} (1991), no. 3, 298--300.

    \bibitem{CH18}
        J. Choe and J. Hoppe, \emph{Some minimal submanifolds generalizing the Clifford torus}, Math.
        Nachr. \textbf{291} (2018), no. 17--18, 2536--2542.

    \bibitem{CM12}
        T. H. Colding and W. P. Minicozzi II, \emph{Generic mean curvature flow I: generic singularities},
        Ann. of Math. (2) \textbf{175} (2012), no. 2, 755--833.

    \bibitem{CLOT03}
        O. Cornea, G. Lupton, J. Oprea and D. Tanr\'e, \emph{Lusternik--Schnirelmann Category}, Mathematical Surveys and Monographs \textbf{103},
        American Mathematical Society, Providence, RI, 2003.

    \bibitem{DU05}
        F. Dobarro and B. \"Unal, \emph{Curvature of multiply warped products}, J.
        Geom. Phys. \textbf{55} (2005), 75--106.

    \bibitem{EM08}
        N. Ejiri and M. Micallef, \emph{Comparison between second variation of area and second variation of energy of a minimal surface},
        Adv. Calc. Var. \textbf{1} (2008), no. 3, 223--239.

    \bibitem{GKS04}
        S. Goette, N. Kitchloo and K. Shankar, \emph{Diffeomorphism type of the Berger space $SO(5)/SO(3)$},
        Amer. J. Math. \textbf{126} (2004), no. 2, 395--416.

    \bibitem{HL71}
        W.-Y. Hsiang and H. B. Lawson, Jr., \emph{Minimal submanifolds of low cohomogeneity}, J.
        Differential Geom. \textbf{5} (1971), 1--38.

    \bibitem{Hui90}
        G. Huisken, \emph{Asymptotic behavior for singularities of the mean curvature flow}, J.
        Differential Geom. \textbf{31} (1990), no. 1, 285--299.

    \bibitem{JRT26}
        J. Jones, D. Rumynin and A. R. Thomas, \emph{The classical topological invariants of homogeneous spaces},
        Doc. Math. (2026), published online first, \href{https://doi.org/10.4171/DM/1085}{doi:10.4171/DM/1085}.

    \bibitem{KM23}
        N. Kapouleas and P. McGrath,
        \emph{Generalizing the linearized doubling approach, I:
        General theory and new minimal surfaces and self-shrinkers},
        Cambridge J. Math. \textbf{11} (2023), no. 2, 299--439.

    \bibitem{KM19}
        N. Kapouleas and P. McGrath, \emph{Minimal surfaces in the round three-sphere by doubling the equatorial two-sphere, II},
        Comm. Pure Appl. Math. \textbf{72} (2019), no. 10, 2121--2195.

    \bibitem{KY10}
        N. Kapouleas and S.-D. Yang, \emph{Minimal surfaces in the three-sphere by doubling the Clifford torus},
        Amer. J. Math. \textbf{132} (2010), no. 2, 257--295.

    \bibitem{KNPS25}
        M. Karpukhin, M. Nahon, I. Polterovich and D. Stern, \emph{Stability of isoperimetric inequalities for Laplace eigenvalues on surfaces},
        J. Differential Geom. \textbf{129} (2025), no. 2, 415--490.

    \bibitem{Law70}
        H. B. Lawson, Jr., \emph{Complete minimal surfaces in $S^3$}, Ann. of Math.
        (2) \textbf{92} (1970), 335--374.

    \bibitem{LZ26}
    \begingroup\raggedright
        H. Li and Y. Zhang, \emph{The geometry and dynamics of spiral minimal products},
        \href{https://arxiv.org/abs/2608.02370}{arXiv:2608.02370} (2026).
    \par\endgroup

    \bibitem{MN14}
        F. C. Marques and A. Neves, \emph{Min-max theory and the Willmore conjecture}, Ann.
        of Math. (2) \textbf{179} (2014), no. 2, 683--782.

    \bibitem{Pit81}
        J. T. Pitts, \emph{Existence and regularity of minimal surfaces on Riemannian manifolds}, Mathematical Notes \textbf{27},
        Princeton University Press, Princeton, NJ, 1981.

    \bibitem{RS26}
        N. Rauchenberger and U. Semmelmann, \emph{The index of cubic focal manifolds},
        \href{https://arxiv.org/abs/2604.09802}{arXiv:2604.09802} (2026).

    \bibitem{Sim68}
        J. Simons, \emph{Minimal varieties in Riemannian manifolds}, Ann. of Math.
        (2) \textbf{88} (1968), no. 1, 62--105.

    \bibitem{Smo05}
        K. Smoczyk, \emph{Self-shrinkers of the mean curvature flow in arbitrary codimension}, Int.
        Math. Res. Not. \textbf{2005} (2005), no. 48, 2983--3004.

    \bibitem{Tak66}
        T. Takahashi, \emph{Minimal immersions of Riemannian manifolds}, J. Math.
        Soc. Japan \textbf{18} (1966), 380--385.

    \bibitem{Tak68}
        F. Takens, \emph{The minimal number of critical points of a function on a compact manifold and the Lusternik--Schnirelman category},
        Invent. Math. \textbf{6} (1968), 197--244.

    \bibitem{TZ20}
        Z. Tang and Y. Zhang, \emph{Minimizing cones associated with isoparametric foliations}, J.
        Differential Geom. \textbf{115} (2020), no. 2, 367--393.

    \bibitem{Whi91}
        B. White, \emph{The space of minimal submanifolds for varying Riemannian metrics}, Indiana Univ.
        Math. J. \textbf{40} (1991), no. 1, 161--200.

    \bibitem{Wiy20}
        D. Wiygul, \emph{Minimal surfaces in the $3$-sphere by stacking Clifford tori}, J.
        Differential Geom. \textbf{114} (2020), no. 3, 467--549.

    \bibitem{Xin03}
        Y. L. Xin, \emph{Minimal submanifolds and related topics}, Nankai Tracts in Mathematics \textbf{8},
        World Scientific, Singapore, 2003.

    \bibitem{Hyp26}
        Y. Zhang, \emph{A generalized Lai--Wei construction and area-minimizing hypercones},
        manuscript (2026).

    \bibitem{SS26}
        Y. Zhang, \emph{Common-source assembly for spherical minimal immersions},
        manuscript (2026).

    \bibitem{SL26}
        Y. Zhang, \emph{Weighted special Legendrian profiles},
        manuscript (2026).

    \end{thebibliography}
\end{document}